\documentclass[12pt]{amsart}
\usepackage{amssymb}
\usepackage{braket}
\usepackage{mathrsfs}
\usepackage[initials]{amsrefs}
\usepackage{ifthen}
\usepackage[all]{xy}
\usepackage{tikz}
\usepackage{comment}

\theoremstyle{plain}
\newtheorem{theo}{Theorem}[section]
\newtheorem{prop}[theo]{Proposition}
\newtheorem{lem}[theo]{Lemma}
\newtheorem{cor}[theo]{Corollary}
\theoremstyle{definition}
\newtheorem{defi}[theo]{Definition}

\newtheorem{rem}[theo]{Remark}

\newtheorem{nota}[theo]{Notation}

\newcommand{\Z}{\mathbb{Z}}
\newcommand{\N}{\mathbb{N}}
\newcommand{\Q}{\mathbb{Q}}
\newcommand{\R}{\mathbb{R}}
\newcommand{\C}{\mathbb{C}}
\newcommand{\KK}{\mathbb{K}}
\newcommand{\OR}{\mathcal{O}}
\newcommand{\DUAL}{\boldmath{D}}
\newcommand{\PR}{\mathbb{P}}
\newcommand{\Pn}{\mathbb{P}^n}
\newcommand{\Poin}{\C[x_1,\dots,x_n]}
\newcommand{\Pozn}{\C[x_0,\dots,x_n]}
\newcommand{\TR}{\mathbb{C^{*}}}
\newcommand{\IC}{\mathrm{IC}}
\newcommand{\VP}{\mathrm{P}}

\newcommand{\MHM}{\mathop{\mathrm{MHM}}}

\newcommand{\wt}[1]{\widetilde{#1}}
\newcommand{\ICV}{{\mathrm{IC}}_{V}}
\newcommand{\ICVH}{{\mathrm{IC}^{H}_{V}}}
\newcommand{\ICVT}{\widetilde{\mathrm{IC}_{V}}}
\newcommand{\ICVTS}{(\widetilde{\mathrm{IC}_{V}})_0}
\newcommand{\ICX}{\mathrm{IC}_{X}}
\newcommand{\ICXH}{{\mathrm{IC}^{H}_{X}}}
\newcommand{\ICXT}{\widetilde{\mathrm{IC}_{X}}}

\newcommand{\RGC}{\mathrm{R}\Gamma_{c}}
\newcommand{\RG}{\mathrm{R}\Gamma}
\newcommand{\DER}{\mathrm{R}}

\DeclareMathOperator{\MOD}{Mod}
\DeclareMathOperator{\SHM}{SHM}
\DeclareMathOperator{\INT}{int}
\DeclareMathOperator{\PERV}{Perv}
\DeclareMathOperator{\RAT}{rat}

\DeclareMathOperator{\Conv}{Conv}
\DeclareMathOperator{\trun}{trun}
\DeclareMathOperator{\supp}{supp}
\newcommand{\REG}{\mathrm{reg}}
\newcommand{\LIM}{\mathrm{lim}}
\newcommand{\DEL}{\mathrm{Del}}

\newcommand{\point}{\mathrm{pt}}
\DeclareMathOperator{\NP}{NP}
\DeclareMathOperator{\DCB}{\mathrm{D}^{\mathrm{b}}}
\DeclareMathOperator{\DCBC}{\mathrm{D}^{\mathrm{b}}_{\mathrm{c}}}
\newcommand{\GR}{\mathrm{gr}}

\AtBeginDocument{%
	\def\MR#1{}
}

\title[On the MHS of IC stalks of complex hypersurfaces]{\Large{O\lowercase{n the mixed} H\lowercase{odge structures of the intersection cohomology stalks of complex hypersurfaces}}}

\author[Takahiro SAITO]{\large{Takahiro SAITO}}

\address{Institute of Mathematics, University of
Tsukuba, 1-1-1, Tennodai,
Tsukuba, Ibaraki, 305-8571, Japan.}
\email{takahiro@math.tsukuba.ac.jp}


\date{\today}

\begin{document}

\maketitle
\begin{abstract}
We consider a hypersurface in $\C^n$ with an isolated singular point at the origin,
and study the mixed Hodge structure of the stalk of its intersection cohomology complex at the origin.
In particular we express the dimension of each graded piece of the weight filtration of this mixed Hodge structure in terms of the numbers of the Jordan blocks in the Milnor monodromy.
\end{abstract}

\tableofcontents

\section{Introduction} 
In this paper, we reveal a new relationship between the mixed Hodge structures of the stalks of the intersection cohomology complexes (we call them IC stalks for short)
and the Milnor monodromies,
by using the results on motivic Milnor fibers shown by Matsui-Takeuchi~\cite{MT}
and on motivic nearby fibers by Stapledon~\cite{SFM}.

For a natural number $n\geq 2$, let $f\in\Poin$ be a non-constant polynomial of $n$ variables with coefficients in $\C$ such that $f(0)=0$.
Assume that $f$ is convenient and non-degenerate at $0$ (see Definitions~\ref{convenient2} and \ref{nondegatzero}). 
We denote by $V$ the hypersurface $\{x\in\C^n\ |\ f(x)=0\}$ in $\C^n$.
Then, it is well-known that $0\in V$ is a smooth or isolated singular point of $V$.
We denote by $\ICV:=\ICV(\Q)$ the intersection cohomology complex of $V$ with rational coefficients.
This is the underlying perverse sheaf of the mixed Hodge module $\ICVH$.
By Morihiko Saito's theory, the stalk $(\ICV)_0$ of $\ICV$ at $0$ is a complex of mixed Hodge structures.
However, to the best of our knowledge, the mixed Hodge structure of $(\ICV)_{0}$ has not been fully studied yet.
For a complex $C$ of mixed Hodge structures and integers $r,k\in\Z$,
we denote by $\GR^{W}_{r}H^{k}(C)$ the $r$-th graded piece of $H^{k}(C)$ with respect to the weight filtration $W_{\bullet}$.
For an integer $w\in\Z$, we say that $C$ has a pure weight $w$ if $\GR^{W}_{r}H^{k}(C)=0$ for any $r,k\in\Z$ with $r\neq k+w$. 
The purity of the weights of IC stalks are important.
Kazhdan-Lustzig computed the Kazhdan-Lustzig polynomials by using the purity of IC stalks of Schubert varieties in flag varieties
in \cite{KL}.
Denef-Loeser proved that if $f$ is quasi-homogeneous, $(\ICV)_0[-(n-1)] (=:(\ICVT)_0)$ has a pure weight $0$ in \cite{DL} (see~Proposition~\ref{dlpure}).
By using this result, 
they computed the dimensions of the intersection cohomology groups of complete toric varieties.

In general, $(\ICVT)_{0}$ has mixed weights $\leq 0$, that is $\GR^{W}_{r}H^{k}((\ICVT)_{0})=0$ for $r>k$.
In this paper, we will describe the dimensions of $\GR^{W}_{r}H^{k}((\ICVT)_{0})$ very explicitly.
We denote by $N_{0}$ the dimension of the invariant subspace of the $(n-1)$-st Milnor monodromy $\Phi_{0,n-1}$ of $f$ at $0$.
Assuming that $n\geq 3$, for any $k\in\Z$, we have
\renewcommand{\arraystretch}{1.5}
\begin{align*}
\dim H^{k}(\ICVTS)=
\left\{
\begin{array}{ll}
1&\text{ if \ }k=0,\\
N_{0}&\text{ if \ }k=n-2, \text{ and}\\
0&\ \text{otherwise}
\end{array}
\right.
\end{align*}
(see Proposition~\ref{icstalkjigen}).\renewcommand{\arraystretch}{1}
Thus, 
if $H^{n-2}(\ICVTS)$ does not have a pure weight,
the dimension $N_{0}$ is decomposed into those of the graded pieces ${\GR^{W}_{r}H^{n-2}((\ICVT)_{0})}$.
We shall describe $\dim{\GR^{W}_{r}H^{n-2}((\ICVT)_{0})}$ in terms of the numbers of the Jordan blocks for the eigenvalue $1$ in the $(n-1)$-st Milnor monodromy $\Phi_{0,n-1}$.
For a natural number $s\in\Z_{\geq 0}$, we denote by $J^{1}_s$ the number of the Jordan blocks in $\Phi_{0,n-1}$ for the eigenvalue 1 with size $s$.
Then our main result is the following.

\begin{theo}[Theorem~\ref{mainth}]\label{intromain}
Assume that $n\geq 3$ and $f$ is convenient and non-degenerate at $0$.
Then for any $r\in\Z$, we have\renewcommand{\arraystretch}{1.5} 
\begin{align*}
\dim \GR^{W}_{r}H^{0}((\ICVT)_0) =
\left\{
\begin{array}{ll}
1 & \text{ if \ } r=0, \\
0 & \text{ if \ } r\neq 0, 
\end{array}
\right.
\end{align*}\renewcommand{\arraystretch}{1}
and 
\[\dim \GR^{W}_{r}H^{n-2}((\ICVT)_0) = J^{1}_{n-r-1}.\]

\end{theo}
For the case where $n=2$, see Theorem~\ref{n2case}.
In particular, we obtain the following result on the purity of the IC stalk $(\ICVT)_{0}$.

\begin{cor}[Corollary~\ref{mainco}]
In the situation of Theorem~\ref{intromain}, the following conditions are equivalent.
\begin{enumerate}
\item The IC stalk $\ICVTS$ has a pure weight $0$.
\item There is no Jordan block for the eigenvalue $1$ with size $>1$ of the $(n-1)$-st Milnor monodromy $\Phi_{0,n-1}$.
\end{enumerate}
\end{cor}

Therefore, the result on the purity of the IC stalks for quasi-homogeneous polynomials by Denef-Loeser (Proposition~\ref{dlpure}) is a special case of our result.

Moreover, as a corollary of the above theorem, we obtain a result on the mixed Hodge structures of the cohomology groups of the link of the isolated singular point $0$ in $V$.
We denote by $L$ the link of $0$ in $V$, that is, the intersection of $V$ and a sufficiently small sphere centered at $0$. 
Then $L$ is a $(2n-3)$-dimensional orientable compact real manifold and each cohomology group $H^{k}(L;\Q)$ of $L$ has a canonical mixed Hodge structure.
For any $k\leq n-2$, the $k$-th cohomology group $H^{k}(L;\Q)$ is isomorphic (as mixed Hodge structures) to $H^{k}(\ICVTS)$.
Assuming that $n\geq 3$, by the Poincar\'{e} duality, we have
$H^{k}(L;\Q)=0$ for any $k\neq0,n-2,n-1,2n-3$.
Then we obtain the following result.
\renewcommand{\arraystretch}{1.5}
\begin{cor}[Corollary~\ref{corlink}]
In the situation of Theorem~\ref{intromain}, for any $r\in\Z$, we have
\begin{align*}
\dim{\GR^{W}_{r}H^{0}(L;\Q)}=\dim{\GR^{W}_{2(n-1)-r}H^{2n-3}(L;\Q)}=\left\{
\begin{array}{ll}
1&\text{ if \ }r=0,\\
0&\text{ if \ }r\neq 0,
\end{array}
\right.
\end{align*}
and
\[\dim{\GR^{W}_{r}H^{n-2}(L;\Q)}=\dim{\GR^{W}_{2(n-1)-r}H^{n-1}(L;\Q)}=J^{1}_{n-r-1}.\]
\end{cor}\renewcommand{\arraystretch}{1}

We prove Theorem~\ref{intromain} by combining the results of Stapledon~\cite{SFM} with that of Matsui-Takeuchi~\cite{MT}.
For this purpose, we deform the hypersurface $V$ to a suitable one which can be compactified in $\Pn$ nicely.
See Section~\ref{sechenkei} for the details.

\subsection*{Acknowledgments}
The author would like to express his hearty gratitude to Professor Kiyoshi Takeuchi for drawing the author's attention to this problem and the several discussions for this work.
His thanks goes also to Yuichi Ike for many discussions and answering to his questions on sheaf theory,
to Tatsuki Kuwagaki for answering to his questions on algebraic varieties and toric varieties,
and to Tomohiro Asano for discussions on matters in Section~\ref{sechenkei}.

\section{Intersection cohomology}\label{ickaisetu}
In this section, we introduce some basic notations and recall intersection cohomology theory.  
We follow the notations of \cite{KS} and \cite{HTT} (see also \cite{DIMCA} and \cite{SCHURMANN}).
Let $X$ be an algebraic variety over $\C$ and $\KK$ a field.
We denote by $\KK_{X}$ the constant sheaf on $X$ with stalk $\KK$
and by $\DCB(\KK_{X})$ or $\DCB(X)$ the bounded derived category of sheaves of $\KK_{X}$-modules on $X$.
For $F \in \DCB(X)$ and an integer $d\in\Z$, we denote by $F[d]$ the shifted complex of $F$ by the degree $d$.  
Moreover, we denote by $\tau^{\geq d}F, \tau^{\leq d}F$ the truncated complexes of $F$.
We denote by $\DCBC(X)$ the full triangulated subcategory of $\DCB(X)$
consisting of complexes whose cohomology sheaves are constructible.
For a morphism $f\colon X\to Y$ of algebraic varieties,
one can define Grothendieck's six operations $\DER{f}_{*}$, $\DER{f}_{!}$, $f^{-1}$, $f^{!}$, $\overset{\mathrm{L}}{\otimes}$ and $\DER{\mathcal{H}om}$ as functors of derived categories of sheaves.
Let $g:X\to \C$ be a morphism of algebraic varieties.
Denote by $\wt{\TR}$ the universal covering of $\C^{*}$ and by $p\colon\wt{\TR}\simeq \{x\in \C\ |\ \mathrm{Im}(x)>0\}\to\C^{*}, x\mapsto \exp(2\pi\sqrt{-1}x)$ the covering map.
Consider the diagram:
\[
\xymatrix{
&  &(X\setminus g^{-1}(0))\times_{\TR}\wt{\TR}\ar[d]_{\pi} \ar[r]^{\qquad\quad{\pi}'} \ar@{}[rd]|\Box &\wt{\TR} \ar[d]^{p} \\
g^{-1}(0) \ar@{^{(}-{>}}[r]_{i} & X & X\setminus g^{-1}(0) \ar@{_{(}-{>}}[l]^{j\quad}  \ar[r]_{g} &\TR~ , \\  
}
\]
where the maps $i$ and $j$ are the inclusion maps and the square $\Box$ on the right is Cartesian.
For $F\in \DCBC(X)$, we define $\psi_{g}(F)\in \DCBC(g^{-1}(0))$ as $i^{-1}\DER(j\circ \pi)_{*}(j\circ \pi)^{-1}F\in \DCBC(g^{-1}(0))$
and $\phi_{g}(F)\in \DCBC(g^{-1}(0))$ as the mapping cone of the morphism $F|_{g^{-1}(0)}\to\psi_{g}(F)$.
Then we obtain the functors $\psi_{g} , \phi_{g}:\DCBC(X)\to \DCBC(g^{-1}(0))$.
We call $\psi_{g}$ and $\phi_{g}$ the \textit{nearby cycle} and the \textit{vanishing cycle functors}, respectively. 

In this paper, for a $\KK$-vector space $H$, we denote by $H^{*}$ the dual $\KK$-vector space of $H$.
If $X$ is a purely $n$-dimensional smooth variety,
there exist canonical isomorphisms
\[H^{k}(X;\KK) \simeq (H^{2n-k}_{c}(X;\KK))^{*}\]
for any $k\in\Z$ by the Poincar\'e duality.
However, if $X$ has some singular points, we can not expect such a nice symmetry in its cohomology groups $H^{k}(X;\KK)$.
Intersection cohomology theory was invented in order to overcome this problem.
Nowadays this theory is formulated in terms of perverse sheaves.

Let $X$ be an algebraic variety over $\C$ of pure dimension $n$.
We denote by $\PERV(X)(:=\PERV(\KK_{X}))$ the category of perverse sheaves over the field $\mathbb{K}$ on $X$.
This is a full abelian subcategory of $\DCBC(X)$.
Let $U$ be the regular part $X_{\REG}$ of $X$.
Then $\KK_{U}[n]$ is a perverse sheaf on $U$.
Let $j \colon U \hookrightarrow X$ be the inclusion map.
Then there exists a natural morphism $j_{!}\KK_{U}[n] \to \DER{j}_{*}\KK_{U}[n]$ in $\DCBC(X)$.
Taking the $0$-th perverse cohomology groups of both sides, we obtain a morphism  ${}^p j_{!}\KK_{U}[n]\to {}^{p}\DER{j}_{*}\KK_{U}[n]$ in $\PERV(X)$.

\begin{defi}
We define $\ICX(\KK)(=:\ICX)$ as the image of the morphism
\[{}^p j_{!}\KK_{U}[n]\to {}^{p}\DER{j}_{*}\KK_{U}[n]\]
in the abelian category $\PERV(X)$.
We call $\ICX$ the \textit{intersection cohomology complex} of $X$ with coefficients in $\KK$.
\end{defi}

\begin{defi}
We set
\begin{align*}
IH^{k}(X;\KK)&:=H^{k}(\RG(X;\ICX(\KK)[-n])), \text{ and}
\\IH^{k}_{c}(X;\KK)&:=H^{k}(\RGC(X;\ICX(\KK)[-n]))
\end{align*}
for any $k \in \Z$. 
We call $IH^{k}(X;\KK)$ (resp. $IH^{k}_{c}(X;\KK)$) the \textit{$k$-th intersection cohomology group} of $X$ (resp. the \textit{$k$-th intersection cohomology group with compact supports} of $X$).
\end{defi}

\begin{prop}
In the situation as above, for any $k\in\Z$, we have the generalized Poincar\'e duality isomorphism
\[IH^{k}(X;\KK)\simeq(IH^{2n-k}_{c}(X;\KK))^{*}.\]
\end{prop}

We can express $\ICX$ more concretely as follows.
Take a Whitnety stratification $X=\bigsqcup_{\alpha\in A}X_{\alpha}$ of $X$
and set $X_{d}:=\bigsqcup_{\dim{X_{\alpha}}\leq d}X_{\alpha}$.
Then we have a sequence of closed subvarieties
\[X=X_{n}\supset X_{n-1}\supset \dots \supset X_{0} \supset X_{-1}=\emptyset\]
of $X$.
Set $U_{d}:=X\setminus X_{d-1}$.
Then we have a sequence of inclusion maps
\[U_{n}\overset{j_{n}}\hookrightarrow U_{n-1}\overset{j_{n-1}}\hookrightarrow\dots\overset{j_{2}}\hookrightarrow U_{1}\overset{j_{1}}\hookrightarrow U_{0}=X.\]

\begin{prop}
There is an isomorphism
\[\ICX\simeq(\tau^{\leq -1}{{\DER{j}_{1}}_{*}})\circ(\tau^{\leq -2}{{\DER{j}_{2}}_{*}})\circ \dots \circ(\tau^{\leq -n}{{\DER{j}_{n}}_{*}})(\KK_{U_{n}}[n])\]
in $\DCBC(X)$.
\end{prop}

In particular, we have
\begin{cor}\label{korituic}
Let $V$ be an algebraic variety over $\C$ of pure dimension $n$,  
and assume that $V$ has only one singular point at $p \in V$. 
Then we have an isomorphism
\[\ICV\simeq\tau^{\leq -1}(\DER{j}_{*}\KK_{V_{\REG}}[n]),\]
where $j:V_{\REG}=V\setminus\{p\}\hookrightarrow V$ is the inclusion map.
\end{cor}

From now on, we mainly consider the shifted complex $\ICVT:=\ICV[-n]$ of $\ICV$. 
By Corollary~\ref{korituic}, we have
\[\ICVT \simeq \tau^{\leq n-1}\DER{j}_{*}\KK_{V_{\REG}}.\]

\section{Limit mixed Hodge structures and mixed Hodge modules}
In this section, we recall basic properties on limit mixed Hodge structures and mixed Hodge modules
(see \cite{MHM}, \cite{CSH}, \cite[Chapter~8]{HTT}, etc. for the details).
\subsection{Limit mixed Hodge structures}\label{mhs}
In this subsection, we recall fundamental notations on limit mixed Hodge structures.

We denote by $\SHM$ the abelian category of mixed Hodge structures,
by $\SHM^{p}$ the full subcategory of $\SHM$ consisting of
\textit{graded-polarizable} mixed Hodge structures.

Recall the Deligne's fundamental theorem on the theory of mixed Hodge structures.
\begin{theo}[Deligne \cite{De}]\label{delmixed}
Let $X$ be an arbitrary algebraic variety over $\C$.
Then $H^{k}(X;\Q)$ and $H^{k}_c(X;\Q)$ have canonical mixed Hodge structures for any $k\in\Z$.
\end{theo}

Note that if $X$ is a smooth projective variety,
these mixed Hodge structures coincide with the usual pure Hodge structures.
We call the mixed Hodge structures in Theorem~\ref{delmixed} \textit{Deligne's mixed Hodge structures}.

We denote by $B(0,\eta)$ the open disc in $\C$ centered at $0$ with radius $\eta$ and
by $B(0,\eta)^{*}$ the punctured open disc $B(0,\eta)\setminus\{0\}$.
Let $f\colon\C^n\to\C$ be a non-constant polynomial map satisfying that $f(0)=0$.
Then for a sufficiently small $\eta>0$, the restriction of  $f$ to $f^{-1}(B(0,\eta)^{*})$ is a locally trivial fibration.
Considering a lift of a path along a small circle around $0$ in $B(0,\eta)^{*}$,
for any $\epsilon\in B(0,\eta)^{*}$
we can define a \textit{monodromy automorphism} 
\[\Psi_{k}\colon H^{k}_{c}(f^{-1}(\epsilon);\C)\overset{\sim}{\to} H^{k}_{c}(f^{-1}(\epsilon);\C)\]
for any $k\in\Z$.
For any $\epsilon\in B(0,\eta)^{*}$,
the $\Q$-vector space $H_{c}^{k}(f^{-1}(\epsilon);\Q)$ has a canonical mixed Hodge structure by Theorem~\ref{delmixed}.
We denote by $F^{\bullet}$ and $W_{\bullet}$ its Hodge and weight filtrations.
On the other hand, 
we can endow $H_{c}^{k}(f^{-1}(\epsilon);\C)$ and $H_{c}^{k}(f^{-1}(\epsilon);\Q)$ with another filtrations $F_{\infty}^{\bullet}$ and $M_{\bullet}$, which are called the \textit{limit Hodge filtration} and the \textit{relative monodromy filtration}, respectively.
Moreover, Steenbrink-Zucker proved the following theorem.
 
\begin{theo}[Steenbrink-Zucker~\cite{STEZUC}]\label{limitmixed}
For any $k\in\Z$, the triple $(H_{c}^{k}(f^{-1}(\epsilon);\Q), F^{\bullet}_{\infty}, M_{\bullet})$ defines
a mixed Hodge structure.
\end{theo}

This mixed Hodge structure in Theorem~\ref{limitmixed} is called the \textit{limit mixed Hodge structure}. 
The relative monodromy filtration $M_{\bullet}$ satisfies the following properties.
The monodromy automorphims $\Psi_{k}$ are decomposed as $\Psi_{k}=\Psi_{k}^{s}\Psi_{k}^{u}$,
where $\Psi_{k}^{s}$ is the semisimple part of $\Psi_{k}$ and $\Psi_{k}^{u}$ is the unipotent part of $\Psi_{k}$.
We denote by $N$ the logalithm operator $\log \Psi_{k}^{u}$  of $\Psi_{k}^{u}$ on $H^{k}_{c}(f^{-1}(\epsilon);\C)$ and for $r\in\Z$ by $N(r)$ the nilpotent operator on $\GR^{W}_{r}H^{k}_{c}(f^{-1}(\epsilon);\C)$ induced by $N$.
We denote by $M(r)_{\bullet}$ the filtration of $\GR^{W}_{r}H^{k}_{c}(f^{-1}(\epsilon);\C)$ induced by $M_{\bullet}$.
Then $M(r)_{\bullet}$ satisfies the following properties:
\begin{enumerate}
\item$M(r)_{2r}=\GR^{W}_{r}H^{k}_{c}(f^{-1}(\epsilon);\C)$,
\item$N(M(r)_{l})\subset M(r)_{l-2}$ \ for any $l\geq 0$, and
\item the map 
\[N^{l}\colon\GR^{M(r)}_{r+l}\GR^{W}_{r}H^{k}_{c}(f^{-1}(\epsilon);\C)\to \GR^{M(r)}_{r-l}\GR^{W}_{r}H^{k}_{c}(f^{-1}(\epsilon);\C)\]
is an isomorphism.
\end{enumerate}
For an eigenvalue $\lambda\in\C$ of $\Psi_{k}$,
we denote by $H^{k}_{c}(f^{-1}(\epsilon);\C)_{\lambda}$ the generalized eigenspace of $\Psi_{k}$ for the eigenvalue $\lambda$.
Then the number of the Jordan blocks in the linear operator on $\GR^{W}_{r}H^{k}_{c}(f^{-1}(\epsilon);\C)$ induced by $\Psi_{k}$ for the eigenvalue $\lambda$ with size $s$
is equal to 
\[\dim{\GR^{M}_{r+1-s}\GR^{W}_{r}H^{k}_{c}(f^{-1}(\epsilon);\C)_{\lambda}}-\dim{\GR^{M}_{r-1-s}\GR^{W}_{r}H^{k}_{c}(f^{-1}(\epsilon);\C)_{\lambda}}.\]

Eventually, $H^{k}_{c}(f^{-1}(\epsilon);\Q)$ has three filtrations $F_{\infty}^{\bullet}, M_{\bullet}$ and $W_{\bullet}$.
We set
\[h^{p,q,r}_{\lambda}(H^{k}_{c}(f^{-1}(\epsilon);\C)):=\dim \GR^{p}_{F_{\infty}}\GR^{M}_{p+q}\GR^{W}_{r}H^{k}_{c}(f^{-1}(\epsilon);\C)_{\lambda}.\]
Then we have the following symmetry:
\[h^{p,q,r}_{\lambda}(H^{k}_{c}(f^{-1}(\epsilon);\C))=h^{r-q,r-p,r}_{\lambda}(H^{k}_{c}(f^{-1}(\epsilon);\C)).\]
The limit mixed Hodge structure encodes the data of the monodromy $\Psi_{k}$ only partially.
In general, the limit mixed Hodge structure does not determine the Jordan normal form of the mondromy $\Psi_{k}$.
However in some good situation as in the following proposition, it completely recovers the Jordan normal form of the monodromy.

\begin{prop}[Stapledon~{\cite[Example~6.7]{SFM}}]\label{syuucyuu}
Assume that $f$ is convenient~(see Definition~\ref{convenient2}) and non-degenerate~(see Definition~\ref{nondeg}).
Then we have $H^{k}_{c}(f^{-1}(\epsilon);\C)=0$ for any $k\in\Z$ with $k\neq n-1,2(n-1)$ and the monodromy action on 
$H^{2(n-1)}_{c}(f^{-1}(\epsilon);\C)\simeq \C$ is trivial.
Moreover, the monodromy automorphism $\Psi_{n-1}$ on the graded piece $\GR^{W}_{r}H^{n-1}_{c}(f^{-1}(\epsilon);\C)$
is trivial for $r\neq n-1$.
Therefore, by the properties of $M_{\bullet}$, for any $r\neq n-1$ and $q\neq r$, we have
\[\GR^{M}_{q}\GR^{W}_{r}H^{n-1}_{c}(f^{-1}(\epsilon);\C)=0.\]
In particular, for any eigenvalue $\lambda\neq 1$ and any $r\neq n-1$, we have
\[\GR^{W}_{r}H^{n-1}_{c}(f^{-1}(\epsilon);\C)_{\lambda}=0,\]
and
\[H^{n-1}_{c}(f^{-1}(\epsilon);\C)_{\lambda}\simeq \GR^{W}_{n-1}H^{n-1}_{c}(f^{-1}(\epsilon);\C)_{\lambda}.\]
\end{prop}

This result follows from a deep argument combining the computation of the motivic nearby fiber and some combinatorial results (see Stapledon~\cite{SFM}, Saito-Takeuchi~\cite{S-T}, Guibert-Loeser-Merle~\cite{GLM}).
This proposition implies that
for an eigenvalue $\lambda\neq 1$, the filtration $M_{\bullet}$ is equal to $M(n-1)_{\bullet}$ on $H^{n-1}_{c}(f^{-1}(\epsilon);\C)_{\lambda}$.
Thus, the Jordan normal form of the monodromy automorphism $\Psi_{n-1}$ for the eigenvalue $\lambda\neq1$ can be completely determined by
the dimensions of the graded pieces $\GR^{M}_{q}H^{n-1}_{c}(f^{-1}(\epsilon);\C)_{\lambda}$ with respect to the monodromy filtration $M_{\bullet}$.

Finally, we define \textit{virtual Poincar\'e polynomials}.
\begin{defi}
Let $C$ be a complex of mixed Hodge structures.
Then we define a Laurent polynomial $\VP(C)(T)$ with coefficients in $\Z$ by
\[\VP(C)(T):=\sum_{r\in\Z}\left(\sum_{k\in\Z}(-1)^{k}\dim \GR^{W}_{r}H^{k}(C)\right)T^{r}~\in\Z[T^{\pm}],\]
where $\GR^{W}_{r}H^{k}(C)$ stands for the $r$-th graded piece with respect to the weight filtration $W_{\bullet}$ of the mixed Hodge structure of $H^{k}(C)$.
We call it the \textit{virtual Poincar\'e polynomial} of $C$.
\end{defi}

Virtual Poincar\'e polynomials satisfy the following property.
\begin{prop}
\begin{enumerate}
\item Let $C\to D\to E\xrightarrow{+1}$ be a distinguished triangle in $\DCB(\SHM)$.
Then we have
\[\VP(C)+\VP(E)=\VP(D).\]
\item Let $C$ and $D$ be elements of $\DCB(\SHM)$.
Then we have
\[\VP(C\otimes D)=\VP(C)\VP(D).\]
\end{enumerate}
\end{prop}

\subsection{Mixed Hodge modules}\label{MHM}
In this subsection, we recall fundamental notations on mixed Hodge modules.
\begin{nota}
Let $X$ be an algebraic variety over $\C$.
We denote by $\MHM(X)$ the abelian category of \textit{mixed Hodge modules} on $X$,
and by $\DCB\MHM(X)$ the bounded derived category of mixed Hodge modules.
Denote by $\RAT\colon \MHM(X) \to \PERV(X)(=\PERV(\Q_{X}))$ the forgetful functor, which assigns a mixed Hodge module to its underlying perverse sheaf.
It induces a functor from $\DCB\MHM(X)$ to $\DCB(\PERV(X))\simeq \DCBC(X)$, which we denote by the same symbol $\RAT$.
\end{nota}

We denote by $\mathrm{MF}_{\mathrm{\mathrm{rh}}}(D_{X},\Q)$ the category of triples $(\mathcal{M},F^{\bullet},K)$ consisting of a regular holonomic $D_{X}$-module $\mathcal{M}$, a good filtration $F^{\bullet}$ of $\mathcal{M}$ and a perverse sheaf $K\in\PERV(X)$ such that the image of $\mathcal{M}$ by the de Rham functor is $\C\otimes_{\Q} K$.
We write $\mathrm{MF_{rh}W}(D_{X},\Q)$ for the category of quadruples $(\mathcal{M},F^{\bullet},K,W_{\bullet})$ of $(\mathcal{M},F^{\bullet},K) \in \mathrm{MF}_{\mathrm{rh}}(D_{X},\Q)$
and its finite increasing filtration $W_{\bullet}$.
Recall that if $X$ is smooth, $\MHM(X)$ is an abelian subcategory of $\mathrm{MF_{rh}W}(D_{X},\Q)$.
Even if $X$ is singular, $\MHM(X)$ can be defined by using a locally embedding of $X$ into a smooth variety.
The category of mixed Hodge modules on the one point variety $\mathrm{pt}$ is equivalent to that of graded-polarizable mixed Hodge structures ${\SHM}^{p}$.
Let $f\colon X\to Y$ and $g\colon X\to \C$ be morphisms of algebraic varieties.
Then we have functors between derived categories of mixed Hodge modules,
$f_{*}$, $f^{*}$, $f_{!}$, $f^{!}$, $\psi_{g}$, $\phi_{g}$, $\DUAL$, $\otimes$ and $\mathcal{H}om$,
and all of them are compatible with the corresponding functors for constructible sheaves via the functor $\RAT$.
For example, we have $\RAT\circ f_{*}= \DER{f}_{*}\circ \RAT$ as a functor from $\DCB\MHM(X)$ to $\DCBC(Y)$.

The following definition is important.

\begin{defi}
Let $\mathcal{V} \in \DCB\MHM(X)$.
Then for any $w\in\Z$,
we say that \textit{$\mathcal{V}$ has mixed weights $\leq w$} ({\rm resp.} $\geq w$) if $\GR^{W}_{r}H^{k}(\mathcal{V})=0$ for any $r,k\in\Z$ with $r>k+w$ ({\rm resp.}\ $r<k+w$),
where $\GR^{W}_{r}H^{k}(\mathcal{V})$ is the $r$-th graded piece of $H^{k}(\mathcal{V})$ with respect to its weight filtration.
Moreover we say that \textit{$\mathcal{V}$ has a pure weight $w$} if $\GR^{W}_{r}H^{k}(\mathcal{V})=0$ for any $r\neq k$.
\end{defi}

We have the following relation between weights and functors of derived categories of mixed Hodge modules.
\begin{prop}[M. Saito~\cite{MHM}]\label{weight}
Let $X$, $Y$ and $Z$ be algebraic varieties and $f\colon X\to Y$ and $g\colon Z\to X$ be morphisms of algebraic varieties.
Let $\mathcal{V}\in \DCB\MHM(X)$.
Then we have the following:
\begin{enumerate}
\item If $\mathcal{V}$ has mixed weights $\leq w$, then $f_{!}\mathcal{V}$ and $g^{*}\mathcal{V}$ have mixed weights $\leq w$.
\item If $\mathcal{V}$ has mixed weights $\geq w$, then $f_{*}\mathcal{V}$ and $g^{!}\mathcal{V}$ have mixed weights $\geq w$.
\end{enumerate} 
\end{prop}

In particular, if $f$ is a proper morphism, 
the functor $f_{*}=f_{!}$ preserves the purity of $\mathcal{V}$.

We denote by $\point$ the one point variety.
Let $X$ be an algebraic variety and $a_{X}\colon X\to \point$ a morphism from $X$ to $\point$.
We denote by $\Q^{H}_{\point}$ the trivial Hodge module on $\point$,
and set $\Q^{H}_{X}:=a_{X}^{*}\Q^{H}_{\point}$.
Since the image of ${a_{X}}_{*}\Q^H_{X}$ by $\RAT$ is $\RG(X;\Q_{X})$, $H^{k}(X;\Q)$ has a mixed Hodge structure for any $k\in\Z$.
This mixed Hodge structure coincides with the one in Theorem~\ref{delmixed}.
Similarly, by considering ${a_{X}}_{!}\Q^{H}_{X}$, $H_{c}^{k}(X;\Q)$ has a mixed Hodge structure, which coincides with the one in Theorem~\ref{delmixed} for any $k\in\Z$.

Let $f\colon\C^n\to\C$ be a non-constant polynomial map satisfying that $f(0)=0$.
For a sufficiently small $\eta>0$, 
the restriction of $f$ to $f^{-1}(B(0,\eta)^{*})$ is a locally trivial fibration.
This implies that each cohomology sheaf of $\DER{f}_{!}\Q_{\C^{n}}$ is a locally constant sheaf on $B(0,\eta)^{*}$
whose stalk at each point $\epsilon\in B(0,\eta)^{*}$ is isomorphic to $\RGC(f^{-1}(\epsilon);\Q_{f^{-1}(\epsilon)})$.
Therefore, as an object in $\DCB(\MOD(\Q))$, the complex $\psi_{t}\DER{f}_{!}\Q_{\C^{n}}$ is isomorphic to $\RGC(f^{-1}(\epsilon);\Q)$
for a sufficiently small $\epsilon\in B(0,\eta)^{*}$,
where $\psi_{t}$ stands for the nearby cycle functor of the coordinate $t$ of $\C$.
Note that $\psi_{t}\DER{f}_{!}\Q_{\C^n}$ is the underlying constructible sheaf of 
the complex of mixed Hodge modules $\psi_{t}\DER{f}_{!}\Q^{H}_{\C^{n}}$.
Thus, each cohomology group $H^{k}_{c}(f^{-1}(\epsilon);\Q)$ of $\RGC(f^{-1}(\epsilon);\Q)(\simeq \psi_{t}\DER{f}_{!}\Q_{\C^{n}})$
has a mixed Hodge structure.
Moreover these mixed Hodge structures of $H^{k}_{c}(f^{-1}(\epsilon);\Q)$ coincide with the limit mixed Hodge structures introduced in Section~\ref{mhs}.

Let $X$ be an algebraic variety over $\C$.
Set $n:=\dim X$.
Then there is an object in $\MHM(X)$ whose underlying perverse sheaf is $\ICX$.
We denote it by $\ICXH$.
Note that ${a_{X}}_{*}\ICXH[-n]$ is a complex of mixed Hodge structures whose image by the functor $\RAT$ is $\RG(X;\ICX[-n])$.
Therefore, $IH^{k}(X;\Q)$ has a mixed Hodge structure.
The same thing holds also for $IH^{k}_{c}(X;\Q)$.
By the construction, the mixed Hodge module $\ICXH$ has a pure weight $n$.
Hence $\ICXH[-n]$ has a pure weight $0$.
According to Proposition~\ref{weight}, $\RG(X;\ICX[-n])$ has mixed weights $\geq 0$,
and we have $\GR^{W}_{r}IH^{k}(X;\Q)=0$ for $r<k$.
If $X$ is complete, we have ${a_{X}}_{*}={a_{X}}_{!}$,
and hence $\RG(X;\ICX[-n])$ has a pure weight $0$ by Proposition~\ref{weight}.
Therefore, we have $\GR^{W}_{r}IH^{k}(X;\Q)=0$ for $r\neq k$,
and $IH^{k}(X;\Q)$ has a pure Hodge structure of weight $k$.
Moreover, the virtual Poincar\'e polynomial of $\RG(X;\ICX[-n])$ is
explicitly written as follows:
\[\VP(\RG(X;\ICX[-n]))(T)=\sum_{k\in\Z}(-1)^{k}\dim IH^{k}(X;\Q) T^{k}.\]

Fix a point $p$ in $X$.
We denote by $i_{p}\colon \{p\}\hookrightarrow X$ the inclusion map.
Then the image of $i_{p}^{*}\ICXH[-n]$ by the functor $\RAT$ is $(\ICXT)_{p}$,
and has a mixed Hodge structure.
By Proposition~\ref{weight}, it has mixed weights $\leq 0$.
This property will play an important role in the proof of Theorem~\ref{mainth}.

\section{Milnor monodromies and IC stalks}
In this section, we recall some basic facts on Milnor monodromies and
the mixed Hodge structures of the stalks of intersection cohomology complexes.

\subsection{Milnor monodromies}\label{milmono}
For a natural number $n\geq2$, let $f\in\C[x_1,\dots,x_n]$ be a non-constant polynomial of $n$ variables with coefficients in $\C$ such that $f(0)=0$.
For any positive real number $r>0$ and any natural number $m\geq1$, we denote by $B(0,r)$ the open ball in $\C^{m}$ centered at $0$ with radius $r$.

\begin{theo}[Milnor~\cite{MIL}]\label{milnorfiber}
There exists $\epsilon>0$ such that for a sufficiently small $(\epsilon\gg)\eta>0$,
the restriction of $f$ 
\[f\colon B(0,\epsilon)\cap f^{-1}(B(0,\eta)^{*})\to B(0,\eta)^{*}\]
is a locally trivial fibration.
Moreover, if $0$ is an isolated singular point of $f^{-1}(0)$,
its fiber is homotopic to a bouquet of some $(n-1)$-spheres.
\end{theo}

This fibration is called the \textit{Milnor fibration} of $f$ at $0$ and its fiber $F_{0}$ is called the \textit{Milnor fiber of $f$ at $0$}.
We denote by $\mu$ the number of $(n-1)$-spheres in the bouquet which is homotopic to $F_{0}$ and call it the \textit{Milnor number}.
By this theorem, if $0$ is an isolated singular point of $f^{-1}(0)$, then for any $k\in\Z$, we have\begin{align*}
H^{k}(F_{0};\Q)\simeq\left\{\renewcommand{\arraystretch}{1.5}
\begin{array}{ll}
\Q&\text{ if\ \ }k=0,\\
\Q^\mu&\text{ if\ \ }k=n-1,\text{ and}\\
0&\text{ if\ \ }k\neq 0,n-1.
\end{array}\right.\renewcommand{\arraystretch}{1}
\end{align*}
Considering a path along a small circle around the origin in $\TR$,
we obtain an automorphism of $F_{0}$ called the \textit{geometric Milnor monodromy}.
It induces an automorphism
\[\Phi_{0,k}\colon H^{k}(F_0;\C)\overset{\sim}{\to} H^{k}(F_0;\C)\]
for any $k\in\Z$.
This automorphism is called the \textit{$k$-th Milnor monodromy} of $f$ at $0$.
In this paper, when $0$ is an isolated singular point of $f^{-1}(0)$,
we call $\Phi_{0,n-1}$ the \textit{Milnor monodromy} of $f$, for short.
The following well-known fact for $\Phi_{0,n-1}$ is called the \textit{monodromy theorem}.
\begin{theo}\label{month}
Assume that $0$ is an isolated singular point of $f^{-1}(0)$.
Then any eigenvalue of $\Phi_{0,n-1}$ is a root of unity.
Moreover, the maximal size of the Jordan blocks in $\Phi_{0,n-1}$ for the eigenvalue $\lambda\neq 1$ {\rm(resp}. $1${\rm)} is bounded by $n$ {\rm(resp}. $n-1${\rm)}.
\end{theo}

In what follows, we assume that $0$ is an isolated singular point of $f^{-1}(0)$.
Then by a theorem of Steenbrink~\cite{STEVAN},
$H^{n-1}(F_0;\Q)$ has a canonical mixed Hodge structure, called the \textit{limit mixed Hodge structure}.
Recall that we defined the nearby cycle functor $\psi_{f}$ in Section~\ref{ickaisetu}.
We denote by $i_{0}\colon \{0\}\hookrightarrow f^{-1}(0)$ the inclusion map.
One can show that for any integer $k\in\Z$, the $k$-th cohomology group of the complex of sheaves $i_{0}^{-1}\psi_{f}(\Q_{\C^{n}})$
is isomorphic to $H^{k}(F_{0};\Q)$.
Since $i_{0}^{-1}\psi_{f}(\Q_{\C^{n}})$ is the underlying complex of that of
mixed Hodge structures $i_{0}^{*}\psi_{f}(\Q_{\C^n}^{H})$,
the $\Q$-vector space $H^{k}(i_{0}^{*}\psi_{f}(\Q_{\C^n}^{H}))\simeq H^{k}(F_{0};\Q)$ has a canonical mixed Hodge structure for any $k\in\Z$.
These mixed Hodge structures coincide with Steenbrink's limit mixed Hodge structures.
We denote by $F^{\bullet}$ its Hodge filtration and $W_{\bullet}$ its weight filtration.
For $\lambda\in\C$, let $H^{n-1}(F_0;\C)_{\lambda}\subset H^{n-1}(F_{0};\C)$ be the generalized eigenspace of $\Phi_{0,n-1}$ for the eigenvalue $\lambda$
and set
\[h^{p,q}_{\lambda}(H^{n-1}(F_0;\C)):=\dim \GR^{p}_{F}\GR^{W}_{p+q}H^{n-1}(F_0;\C)_{\lambda}.\]
For these numbers, the following results are well-known (see Steenbrink~\cite{STEVAN}).

\begin{prop}\label{mtsansyou}
$(i)$~For $\lambda\in\C^{*}\setminus\{1\}$ and $(p,q)\notin[0,n-1]\times[0,n-1]$
we have $h^{p,q}_{\lambda}(H^{n-1}(F_0;\C)) = 0$.
For $(p,q)\in[0,n-1]\times[0,n-1]$ we have
\[h^{p,q}_{\lambda}(H^{n-1}(F_0;\C))=h^{n-1-q,n-1-p}_{\lambda}(H^{n-1}(F_0;\C)).\]
$(ii)$~For $(p,q)\notin[1,n-1]\times[1,n-1]$ we have $h^{p,q}_{1}(H^{n-1}(F_0;\C))=0$.
For $(p,q)\in[1,n-1]\times[1,n-1]$ we have
\[h^{p,q}_{1}(H^{n-1}(F_0;\C))=h^{n-q,n-p}_{1}(H^{n-1}(F_0;\C)).\]
\end{prop}

\begin{prop}\label{lmhtojordan}
$(i)$~For $\lambda\in\C^{*}\setminus\{1\},s\geq 1$,
the number of the Jordan blocks in $\Phi_{0,n-1}$ with size $\geq s$ for the eigenvalue $\lambda$ is equal to
\[\sum_{p+q=n-2+s,n-1+s}h^{p,q}_{\lambda}(H^{n-1}(F_0;\C)).\]
$(ii)$~For $s\geq 1$, 
the number of the Jordan blocks in $\Phi_{0,n-1}$ with size $\geq s$ for the eigenvalue $1$ is equal to
\[\sum_{p+q=n-1+s,n+s}h^{p,q}_{1}(H^{n-1}(F_0;\C)).\]
\end{prop}
In what follows, for any $\lambda\in\C^*$ and $s\in\Z$, we denote by $J^{\lambda}_{s}$ the number of the Jordan blocks in $\Phi_{0,n-1}$ with size $s$ for the eigenvalue $\lambda$.
By the monodrmy theorem, if $\lambda\neq 1$ (resp. $\lambda=1$)
we have $J^{\lambda}_{s}=0$ for any $s\geq n+1$ (resp. $s\geq n$).

By Propositions~\ref{mtsansyou} and \ref{lmhtojordan}, we can describe the dimension of each graded piece $\GR^{W}_{r}H^{n-1}(F_{0};\C)_{\lambda}$
of $H^{n-1}(F_{0};\C)_{\lambda}$ with respect to the weight filtration as follows.
\begin{prop}\label{milnohyou}
$(i)$~For $\lambda\in\C^{*}\setminus\{1\}$ and $r<0$ or $2(n-1)<r$, we have $\GR^{W}_{r}H^{n-1}(F_{0};\C)_{\lambda}=0$.
For $0\leq r \leq 2(n-1)$, we have the symmetry
\[\dim \GR^{W}_{2(n-1)-r}H^{n-1}(F_{0};\C)_{\lambda}=\dim \GR^{W}_{r}H^{n-1}(F_{0};\C)_{\lambda},\]
centered at degree $n-1$.
Moreover, for $0\leq r \leq n-1$, we have
\[\dim \GR^{W}_{r}H^{n-1}(F_{0};\C)_{\lambda} = \sum_{s\geq0}J^{\lambda}_{n-r+2s}.\]

$(ii)$~For $r<2$ or $2(n-1)<r$, we have $\GR^{W}_{r}H^{n-1}(F_{0};\C)_{1}=0$.
For $0\leq r \leq 2(n-1)$, we have the symmetry
\[\dim \GR^{W}_{2(n-1)-r}H^{n-1}(F_{0};\C)_{1}=\dim^{W}_{r}H^{n-1}(F_{0};\C)_{1},\]
centered at degree $n$.
Moreover, for $2\leq r \leq n$, we have 
\[\dim \GR^{W}_{r}H^{n-1}(F_{0};\C)_{1} = \sum_{s\geq0}J^{1}_{n+1-r+2s}.\]

\end{prop}

By Proposition~\ref{milnohyou} we have the following table for the dimensions of the graded pieces of $H^{n-1}(F_{0};\C)_{\lambda}$.

{\begin{table}[!hbt]\raggedright
\begin{flushleft}
\begin{tabular}{c|cccccc}
$r$&$0$&$1$&$2$&$3$&$\cdots$\\ \hline
$\dim \GR^{W}_{r}H^{n-1}(F_0)_{\lambda}$& $J^{\lambda}_{n}$ & $J^{\lambda}_{n-1}$ & $J^{\lambda}_{n-2}+J^{\lambda}_{n}$ & $J^{\lambda}_{n-3}+J^{\lambda}_{n-1}$
&$\cdots$
\end{tabular}
\end{flushleft}
\begin{flushright}
\begin{tabular}{cccccc}
$n-2$&$n-1$&$n$&$\cdots$&$2(n-1)$\\ \hline
$J^{\lambda}_{2}+J^{\lambda}_{4}+\cdots$&
$J^{\lambda}_{1}+J^{\lambda}_{3}+\cdots$&
$J^{\lambda}_{2}+J^{\lambda}_{4}+\cdots$&$\cdots$& $J^{\lambda}_{n} $

\end{tabular}
\end{flushright}
\caption{The case $\lambda\neq 1$}\label{tab1}
\end{table}
\begin{table}[!hbt]
\begin{flushleft}
\begin{tabular}{c|cccccc}
$r$&$2$&$3$&$4$&$5$&$\cdots$\\ \hline
$\dim \GR^{W}_{r}H^{n-1}(F_0)_{1}$& $J^{1}_{n-1}$ & $J^{1}_{n-2}$ & $J^{1}_{n-3}+J^{1}_{n-1}$ & $J^{1}_{n-4}+J^{1}_{n-2}$
&$\cdots$
\end{tabular}
\end{flushleft}
\begin{flushright}
\begin{tabular}{cccccc}
$n-1$&$n$&$n+1$&$\cdots$&$2(n-1)$\\ \hline
$J^{1}_{2}+J^{1}_{4}+\cdots$&
$J^{1}_{1}+J^{1}_{3}+\cdots$&
$J^{1}_{2}+J^{1}_{4}+\cdots$&$\cdots$& $J^{1}_{n-1} $
\end{tabular}
\end{flushright}
\caption{The case $\lambda=1$}\label{tab2}
\end{table}}

Next, in order to introduce some results for the numbers of the Jordan blocks in $\Phi_{0,n-1}$, we recall some familiar notions about polynomials.
\begin{defi}\label{newtonpoly}
Let $f=\sum_{\alpha\in\Z^n}a_{\alpha}x^{\alpha}\in\C[x^{\pm}_{1},\dots,x^{\pm}_{n}]$ be a Laurent polynomial with coefficients in $\C$.
We call the convex hull of $\supp(f)=\{\alpha\in\Z^n_{\geq0}\ |\ a_{\alpha}\neq 0\}$ the \textit{Newton polytope} of $f$ and denote it by $\NP(f)$.
\end{defi}

\begin{defi}\label{convenient2}
Let $f=\sum_{\alpha\in\Z^n}a_{\alpha}x^{\alpha}\in\C[x_{1},\dots,x_{n}]$ be a polynomial with coefficients in $\C$ such that $f(0)=0$.
\begin{enumerate}
\item We call the convex hull of $\bigcup_{\alpha\in \supp{(f)}}\{\alpha+\R^n_{\geq 0}\}$ the \textit{Newton polyhedron} of $f$ at the origin $0\in\C^{n}$ and denote it by $\Gamma_{+}(f)$.
\item We call the union of all bounded face of $\Gamma_{+}(f)$ the \textit{Newton boundary} of $f$ and denote it by $\Gamma_{f}$.
\item We say that $f$ is \textit{convenient} if $\Gamma_{+}(f)$ intersects
the positive part of each coordinate axis of $\R^{n}$.
\end{enumerate}
\end{defi}

For a Laurent polynomial $f=\sum_{\alpha\in\Z^n}a_{\alpha}x^{\alpha}$ and
a polytope $F$ in $\R^{n}$, set $f_{F}:=\sum_{\alpha\in F}a_{\alpha}x^{\alpha}$.

\begin{defi}\label{nondeg}
Let $f=\sum_{\alpha\in\Z^n}a_{\alpha}x^{\alpha}\in\C[x^{\pm}_{1},\dots,x^{\pm}_{n}]$ be a Laurent polynomial with coefficients in $\C$.
We say that $f$ is \textit{non-degenerate}
if for any face $F$ of $\NP(f)$
the hypersurface $\{x\in(\TR)^n\ |\ f_{F}(x)=0\}$ of $(\TR)^{n}$ is smooth and reduced.
\end{defi}

\begin{defi}\label{nondegatzero}
Let $f=\sum_{\alpha\in\Z_{\geq 0}^n}a_{\alpha}x^{\alpha}\in\C[x_{1},\dots,x_{n}]$ be a polynomial with coefficients in $\C$ such that $f(0)=0$.
We say that $f$ is \textit{non-degenerate at $0\in\C^{n}$}
if for any face $F$ of $\Gamma_{f}$, the hypersurface $\{x\in(\TR)^n\ |\ f_{F}(x)=0\}$ of $(\TR)^{n}$ is smooth and reduced.
\end{defi}

Let $f=\sum_{\alpha\in\Z_{\geq 0}^n}a_{\alpha}x^{\alpha}\in\C[x_{1},\dots,x_{n}]$ be a polynomial with coefficients in $\C$ such that $f(0)=0$.
Let $q_1,...,q_l\in\Z^n$ be the vertices in $\Gamma_{f}\cap \INT(\R^n_{\geq 0})$.
We denote by $d_i\in\Z_{>0}$ the lattice distance between $q_i$ and $0\in\R^n$,
and by $\Pi_{f}$ the number of lattice points in the union of one dimensional faces of $\Gamma_{f}\cap \INT(\R^n_{\geq0})$.

\begin{theo}[\cite{DoornSteen}, \cite{MT}]\label{mtjordan}
Assume that $f$ is convenient and non-degenerate at $0\in\C^n$.
Then we have
\begin{enumerate}
\item for $\lambda\in\TR\setminus\{1\}$, $J^{\lambda}_{n}=\#\{q_{i}\ |\ \lambda^{d_{i}}=1\}$, and
\item $J^{1}_{n-1}=\Pi_{f}$.
\end{enumerate}
\end{theo}
For a related result, see Raibaut~\cite{Raibaut}.

The next theorem was proved by Matsui-Takeuchi \cite{MT}.
\begin{theo}[\cite{MT}]\label{jordangammadekimaru}
If $f$ is convenient and non-degenerate at $0$,
all the numbers $J^{\lambda}_{s}$ are determined from the Newton boundary $\Gamma_{f}$ of $f$.
\end{theo}

\subsection{IC stalks and Milnor monodromies}\label{icandmilnor}
As we mentioned in the introduction, there exists a relationship between IC stalks and Milnor monodromies.
Let $f\in\C[x_1,\dots,x_n]$ be a polynomial of $n$ variables with coefficients in $\C$ such that $f(0)=0$.
We assume that $0$ is an isolated singular point of $V=f^{-1}(0)\subset\C^n$.
We denote by $N_{0}$ the dimension of the invariant subspace of $\Phi_{0,n-1}$ in $H^{n-1}(F_{0};\Q)$.

\begin{prop}[see e.g. {\cite[Lemma~4.3]{NT}}]\label{icstalkjigen}\renewcommand{\arraystretch}{1.5}
If $n\geq 3$, then for any $k\in\Z$, we have
\begin{align*}
\dim H^{k}(\ICVTS)=
\left\{
\begin{array}{ll}
1&\text{ if\ \ } k=0,\\
N_{0}&\text{ if\ \ }k=n-2, \text{ and}\\
0&\text{ otherwise}.
\end{array}
\right.
\end{align*}
If $n=2$, then for any $k\in\Z$, we have
\begin{align*}
\dim H^{k}(\ICVTS)=
\left\{
\begin{array}{ll}
N_{0}+1&\text{ if\ \ }k=0, \textit{ and}\\
0&\text{ otherwise}.
\end{array}
\right.
\end{align*}
\end{prop}\renewcommand{\arraystretch}{1}

As mentioned in Section~\ref{MHM}, 
$\ICVTS$ is a complex of mixed Hodge structures having mixed weights $\leq 0$ (see Proposition~\ref{weight}).
Then for any $r\in\Z$, we have
\begin{align*}
\GR^{W}_{r}H^{0}(\ICVTS)&=0\quad\text{if\ \ }r>0,\text{ and}\\
\GR^{W}_{r}H^{n-2}(\ICVTS)&=0\quad\text{if\ \ }r>n-2.
\end{align*}

\begin{defi}\label{quasi}
Let $f=\sum_{\alpha\in\Z^n_{\geq 0}}x^{\alpha} \in\C[x_{1},\dots,x_{n}]$ be a polynomial with coefficients in $\C$ such that $f(0)=0$.
We say that $f$ is \textit{quasi-homogeneous}
if there exists a natural number $C\in\Z_{>0}$ and $v=(v_{1},\dots,v_{n})\in(\Z_{>0})^n$ such that for any $\alpha$ with $a_{\alpha}\neq 0$, $v\cdot \alpha=C$.
\end{defi}

The following theorem was proven by Denef-Loeser~\cite{DL}
to compute the dimensions of the intersection cohomology groups of complete toric varieties.
\begin{theo}[Denef-Loeser~\cite{DL}]\label{dlpure}
Assume that the polynomial $f$ is quasi-homogeneous.
Then, the IC stalk $\ICVTS$ has a pure weight $0$. 
\end{theo}

We introduce the new notion which is a generalization of quasi-homogeneousness.

\begin{defi}\label{flatness}
Let $f$ be a polynomial with coefficients in $\C$ such that $f(0)=0$.
We say that \textit{$\Gamma_{+}(f)$ is flat}
if the affine space spanned by $\Gamma_{f}$ is a hyperplane in $\R^n$ (see Figures~\ref{fig1} and \ref{fig2}).
\end{defi}

For example, the Newton polyhedron $\Gamma_{+}(f)$ of a quasi-homogeneous polynomial $f$ is flat.

\begin{figure}[htbp]
\begin{tabular}{cc}
\begin{minipage}{0.45\hsize}
\begin{center}
\begin{tikzpicture}[scale = 0.5]
\draw[help lines] (-1,-1) grid (6,6);
\draw[fill,black, opacity=.5] (0,4) --(3,0)--(6,0)--(6,6)--(0,6)--cycle;
\draw[->,thick] (0,-1)--(0,6);
\draw[->,thick] (-1,0)--(6,0);
\draw (0,0) node[below left]{(0,0)};
\draw (0,4) node[left]{(0,4)};
\draw (3,0) node[below]{(3,0)};
\draw (4,4) node{$\Gamma_{+}(f)$};
\end{tikzpicture}
\caption{$\Gamma_{+}(f)\mbox{\ is flat}$}\label{fig1}
\end{center}
\end{minipage}&
\begin{minipage}{0.45\hsize}
\begin{center}
\begin{tikzpicture}[scale = 0.5] 
\draw[help lines] (-1,-1) grid (6,6);
\draw[fill,black, opacity=.5] (0,4) --(1,1)--(3,0)--(6,0)--(6,6)--(0,6)--cycle;
\draw[->,thick] (0,-1)--(0,6);
\draw[->,thick] (-1,0)--(6,0);
\draw (0,0) node[below left]{(0,0)};
\draw (0,4) node[left]{(0,4)};
\draw (3,0) node[below]{(3,0)};
\draw (1,1) node[above right]{(1,1)};
\draw (4,4) node{$\Gamma_{+}(f)$};
\end{tikzpicture}
\caption{$\Gamma_{+}(f)\mbox{\ is not flat}$}\label{fig2}
\end{center}
\end{minipage}
\end{tabular}
\end{figure}
Under the assumption that $f$ is convenient and non-degenerate at $0$,
we will show that
the Newton polyhedron $\Gamma_{+}(f)$ completely determines
the mixed Hodge numbers of the mixed Hodge structures of $H^{k}(\ICVTS)$,
in Proposition~\ref{subetedoukei}.

\section{Main theorem}\label{secmainth}
Let $f=\sum_{\alpha\in\Z^n_{\geq 0}}a_{\alpha}x^{\alpha}\in\C[x_1,...,x_n]$ be a non-constant polynomial of $n$ variables with coefficients in $\C$ such that
$f(0)=0$.
Assume that $f$ is convenient (see Definition~\ref{convenient2}) and non-degenerate at $0$ (see Definition~\ref{nondegatzero}).
Set $V:=\{x\in\C^n\ |\ f(x)=0\}\subset \C^n$.
Then it is well-known that $0\in V$ is a smooth or isolated singular point of $V$.
In what follows, we assume that $n\geq 3$.
For the case where $n=2$, see Theorem~\ref{n2case}.

\begin{theo}\label{mainth}
Under the above conditions, for any $r\in\Z$, we have
\begin{align*}
\dim \GR^{W}_{r}H^{0}((\ICVT)_0) =
\left\{\renewcommand{\arraystretch}{1.5}
\begin{array}{ll}
1&\text{ if\ \ }r=0,\\
0&\text{ if\ \ }r\neq 0,
\end{array}
\right.
\end{align*}
and
\[
\dim \GR^{W}_{r}H^{n-2}((\ICVT)_0) = J^{1}_{n-r-1},
\]
where $J^{1}_{s}$ is the number of the Jordan blocks in $\Phi_{0,n-1}$ with size $s$ for the eigenvalue $1$ (see Section~\ref{milmono}).
\end{theo}\renewcommand{\arraystretch}{1}

\begin{cor}\label{mainco}
In the situation of Theorem~\ref{mainth}, the following three conditions are equivalent.
\begin{enumerate}
\item The IC stalk $(\ICVT)_{0}$ has a pure weight $0$.
\item There is no Jordan block of $\Phi_{0,n-1}$ with size $>1$ for the eigenvalue $1$.
\item The Newton polyhedron $\Gamma_{+}(f)$ of $f$ is flat (see Definition~\ref{flatness}).
\end{enumerate}
\end{cor}
\begin{proof}
The equivalence $({\rm i})\iff ({\rm ii})$ follows from Theorem~\ref{mainth}.
Since $n\geq 3$, $({\rm ii})\Rightarrow ({\rm iii})$ follows from Theorem~\ref{mtjordan}.
Suppose that $\Gamma_{+}(f)$ is flat.
We define a quasi-homogeneous polynomial $f'\in\C[x_{1},\dots,x_{n}]$ as
$f':=\sum_{\alpha\in\Gamma_{f}}a_{\alpha}x^{\alpha}$,
and a hypersurface $V'\subset \C^n$ as the zero set of $f'$ in $\C^n$.
Note that $\Gamma_{+}(f')=\Gamma_{+}(f)$ and $f'$ is non-degenerate at $0$.
By Proposition~\ref{subetedoukei} below, we will show that the mixed Hodge numbers of $H^{k}(\ICVTS)$ are the same of
the mixed Hodge numbers of $H^{k}((\wt{\mathrm{IC}_{V'}})_{0})$.
The Milnor monodromy of a quasi-homogeneous polynomial is semisimple.
Therefore we have $({\rm iii})\Rightarrow ({\rm ii})$.
\end{proof}

\begin{rem}
We can rewrite Theorem~\ref{mainth} in terms of the virtual Poincar\'{e} polynomial as follows:
\begin{align*}
\VP(\ICVTS)(T)&=1+(-1)^{n-2}\sum_{i=0}^{n-2}J^{1}_{n-i-1}T^i
\\&=1+(-1)^{n-2}(J^{1}_{n-1}+J^{1}_{n-2}T+J^{1}_{n-3}T^2+\dots+J^{1}_{1}T^{n-2}).
\end{align*}

\end{rem}

\begin{rem}
By Theorem~\ref{month},
there is no Jordan block of $\Phi_{0,n-1}$ with size $\geq n$ for the eigenvalue $1$.
Hence the dimension $N_{0}$ of the invariant subspace of $\Phi_{0,n-1}$ in $H^{n-1}(F_0;\Q)$
is equal to $J^{1}_1+\cdots+J^{1}_{n-1}$. 
On the other hand, by Proposition~\ref{icstalkjigen} and Theorem~\ref{mainth} we have $N_{0}=\dim H^{n-2}(\ICVTS)=\sum_{r\in\Z}\dim \GR^{W}_{r}H^{n-2}(\ICVTS)$.
Namely Theorem~\ref{mainth} means that these two decompositions of $N_{0}$ are the same. 
\end{rem}

As a corollary of Theorem~\ref{mainth}, we get the following result of the mixed Hodge structures of the cohomology groups of the link of the isolated singular point $0$ in $V$.
The intersection of $V$ and a small sphere centered at $0$ is called the \textit{link of $0$ in $V$} and
we denote it by $L$.
It is known that $L$ is a $(2n-3)$-dimensional orientable compact real manifold.
We denote by $j$ the inclusion map $V\setminus\{0\}\hookrightarrow V$ and
by $i_{0}$ the inclusion map $\{0\}\hookrightarrow V$.
Then for any $k\in\Z$, the cohomology group $H^{k}(L;\Q)$ of the link can be expressed as $H^{k}(i_{0}^{-1}\DER{j}_{*}\Q_{V\setminus\{0\}})$.
Note that $H^{k}(i_{0}^{-1}\DER{j}_{*}\Q_{V\setminus\{0\}})$ has a canonical mixed Hodge structure by using the same argument for $H^{k}(\ICVTS)$,
and hence we can endowed $H^{k}(L;\Q)$ with a canonical mixed Hodge structure.
It is clear that for $k\leq n-2$, $H^{k}(\ICVTS)$ is isomorphic to $H^{k}(L;\Q)$ as mixed Hodge structures.
Moreover by \cite[Proposition~3.3]{DurfeeSaito}, for any $k\in\Z$, we have the Poincar\'{e} duality isomorphism as mixed Hodge structures:
\[H^{k}(L;\Q)\simeq (H^{2n-3-k}(L;\Q)(n-1))^{*},\]
where $H^{2n-3-k}(L;\Q)(n-1)$ stands for the $(n-1)$-Tate twist of $H^{2n-3-k}(L;\Q)$.
Therefore, $H^{0}(L;\Q)$, $H^{n-2}(L;\Q)$, $H^{n-1}(L;\Q)$ and $H^{2n-3}(L;\Q)$ are the only non-trivial cohomology groups
and 
for any $r,k\in\Z$, we have
\[\dim{\GR^{W}_{r}H^{k}(L;\Q)}=\dim{\GR^{W}_{2(n-1)-r}H^{2n-3-k}(L;\Q)}.\]
We thus obtain the following.

\begin{cor}\label{corlink}\renewcommand{\arraystretch}{1.5}
In the situation of Theorem~\ref{mainth}, for any $r\in\Z$, we have
\begin{align*}
\dim{\GR^{W}_{r}H^{0}(L;\Q)}=\dim{\GR^{W}_{2(n-1)-r}H^{2n-3}(L;\Q)}=\left\{
\begin{array}{ll}
1&\text{ if\ \ }r=0,\\
0&\text{ if\ \ }r\neq 0,
\end{array}
\right.
\end{align*}
and
\[\dim{\GR^{W}_{r}H^{n-2}(L;\Q)}=\dim{\GR^{W}_{2(n-1)-r}H^{n-1}(L;\Q)}=J^{1}_{n-r-1}.\]
\end{cor}\renewcommand{\arraystretch}{1}

First, we show Theorem~\ref{mainth} in the following very special case.

\begin{lem}\label{mainlem}
Assume that $f$ satisfies the conditions of Theorem~\ref{mainth}.
Moreover, suppose that $f$ has no linear term, its degree $m$ part is $\sum_{1\leq i \leq n}x^m_i$ and $f$ is non-degenerate (see Definition~\ref{nondeg}).
Then, Theorem~\ref{mainth} holds.
\end{lem}

\begin{proof}
Let $g\in\Pozn$ be the projectivization of $f$.
Namely we set $g=x_{0}^{m}f(x_{1}/x_{0},\dots,x_{n}/x_{0})$.
Define a hypersurface $X\subset \Pn$ of $\Pn$ by
$X:=\{x\in\Pn\mid g(x)=0\}$.
Then under the identification $\{[x_0:\cdots:x_n]\in\Pn\mid x_0\neq0\}\simeq\C^n$,
we have $X\cap\C^n=V$.
We decompose $\Pn$ into a disjoint union of small tori as
$\Pn=\bigsqcup_{\emptyset\neq J\subset\{0,\dots,n\}} T_{J}$, where we set 
\[T_{J}:=\{[x_{0} : \dots : x_{n}]\in \Pn \mid x_{i}=0~(i\notin J)~ x_{i}\neq 0~(i\in J)\}\simeq (\TR)^{|J|-1},\]
for $\emptyset\neq J\subset\{0,\dots,n\}$.
By the non-degeneracy of $f$, for any $J\neq \emptyset$, the torus $T_{J}$ intersects $X$ transversally. 
Thus there is only one singular point $p=[1:0:\dots:0]$ in the hypersurface $X\subset\Pn$.

Consider the following distinguished triangle in $\DCB({\SHM}^{p})$:
\begin{align}\label{dt1}
\RGC(X\setminus\{p\};\ICXT)\rightarrow\RGC(X;\ICXT)\rightarrow\RGC(\{p\};\ICXT)\xrightarrow{+1}.
\end{align}
By $(\ICXT)|_{X\setminus\{p\}}=\Q_{X\setminus\{p\}}$, we have
\[\RGC(X\setminus\{p\};\ICXT)=\RGC(X\setminus\{p\};\Q_{X\setminus\{p\}}).\]
Moreover, obviously we have
\[\RGC(\{p\};\ICXT)=\ICVTS.\]
Taking the virtual Poincar\'{e} polynomials of the terms in the distinguished triangle (\ref{dt1}),
we obtain
\begin{align}
\VP(\RGC(X;\ICXT))(T)
&=\VP(\RGC(X\setminus\{p\};\Q))(T)+\VP(\ICVTS)(T)\label{tocyuu}
\\&=\VP(\RGC(X;\Q))(T)-1+\VP(\ICVTS)(T).\notag
\end{align}

Consider the following distinguished triangle in $\DCB({\SHM}^p)$:
\begin{align}\label{dt2}
(\DER{f}_{!}\Q^H_{\C^n})_{0}\rightarrow
\psi_{t}(\DER{f}_{!}\Q^H_{\C^n})\rightarrow
\phi_{t}(\DER{f}_{!}\Q^H_{\C^n})\xrightarrow{+1},
\end{align}
where $t:\C\to\C$ is the identity map.
As objects in $\DCB(\MOD(\Q))$,
the first term is isomorphic to $\RGC(V;\Q)$,
and the second one is isomorphic to $\RGC(V_{\epsilon};\Q)$,
where we set $V_{\epsilon}:=f^{-1}(\epsilon)$ for a sufficiently small $\epsilon>0$. 
By the next lemma, the third one is isomorphic to $H^{n-1}(F_0;\Q)[-(n-1)]$,
where $F_{0}$ is the Milnor fiber of $f$ at $0$ (see Section~\ref{milmono}).

\begin{lem}[cf. \cite{MTWS}]\label{vanishingkakan}
We have an isomorphism in $\DCB(\MOD(\Q))\colon$
\[\phi_{t}(\DER{f}_{!}\Q_{\C^n})\simeq \phi_{f}(\Q_{\C^n})_{0}.\]
\end{lem}
\begin{proof}
Let $\Gamma_{\infty}(f)$ be the convex hull of the union of $\{0\}$ and $\NP(f)$ in $\R^n$.
Let $\Sigma_{0}$ be the fan formed by all faces of $(\R_{\geq 0})^n$ in $\R^n$,
$\Sigma_{1}$ the dual fan of $\Gamma_{\infty}(f)$.
Since $f$ is convenient, $\Sigma_{0}$ is a subfan of $\Sigma_{1}$.
There exists a smooth subdivision $\Sigma$ of $\Sigma_{1}$ which contains $\Sigma_{0}$ as a subfan.
We denote by $X_{\Sigma}$ the toric variety associated to $\Sigma$.
We regard $f$ as an element of the function field of $X_{\Sigma}$,
and eliminate its indeterminacy by blowing up $X_{\Sigma}$ (see \cite[the proof of Theorem~3.6]{MTWS}).
Then we obtain a smooth compact variety $\wt{X_{\Sigma}}$ and a proper morphism $\pi\colon \wt{X_{\Sigma}}\to X_{\Sigma}$
such that $g:=f\circ\pi$ has no point of indeterminacy.
Then $g$ can be considered as a proper map from $\wt{X_{\Sigma}}$ to $\mathbb{P}^{1}$.
The restriction $g^{-1}(\C)\to\C$ of $g$ is also proper,
and we use the same symbol $g$ for this restriction map.
We thus obtain a commutative diagram:
\[
\xymatrix{
\C^n \ar[d]_{f} \ar@{^{(}-{>}}[r]^{\iota\quad}  & g^{-1}(\C) \ar[d]^{g}  \\
\C \ar@{=}[r]^{t}& ~\C~, \\
}
\]
where $\iota$ is the inclusion map.
By \cite[Proposition~4.2.11]{DIMCA} and the properness of $g$,
we obtain the following isomorphism:
\begin{align*}
\RG(g^{-1}(0);\phi_{g}(\iota_{!}\Q_{\C^n}))&\simeq \phi_{t}(\DER{g}_{*}\iota_{!}\Q_{\C^n})
\\&\simeq \phi_{t}(\DER{f}_{!}\Q_{\C^n}).
\end{align*}
It follows from the construction of $\wt{X_{\Sigma}}$
that the support of $\phi_{g}(\iota_{!}\Q_{\C^n})$ does not intersect $\wt{X_{\Sigma}}\setminus\C^n$ (see~\cite[the proof of Theorem~3.6]{MTWS}). 
Moreover, since $V=f^{-1}(0)\subset \C^n$ has only one singular point $0$,
the support of $\phi_{g}(\iota_{!}\Q_{\C^n})$ is \{0\}.
Furthermore, obviously we have $\phi_{g}(\iota_{!}\Q_{\C^n})_{0}\simeq\phi_{f}(\Q_{\C^n})_{0}$.
Combining these results, we finally obtain
\begin{align*}
\phi_{t}(\DER{f}_{!}\Q_{\C^n})\simeq
\RG(g^{-1}(0);\phi_{g}(\iota_{!}\Q_{\C^n}))\simeq\phi_{g}(\iota_{!}\Q_{\C^n})_{0}\simeq\phi_{f}(\Q_{\C^n})_{0}.
\end{align*}
\end{proof}

The $k$-th cohomology group $H^{k}_c(V;\Q)$ of $\RGC(V;\Q)$ has Deligne's mixed Hodge structure (see Theorem~\ref{delmixed}).
The $k$-th cohomology group $H^{k}_c(V_{\epsilon};\Q)$ of $\RGC(V_{\epsilon};\Q)$ has the limit mixed Hodge structure which encodes some informations of the monodromy action
(see Theorem~\ref{limitmixed}).
Moreover the $(n-1)$-st cohomology group $H^{n-1}(F_0;\Q)$ of the Milnor fiber $F_{0}$ of $f$ at $0$ has the limit mixed Hodge structure which encodes some informations of the Milnor monodromy $\Phi_{0,n-1}$ (see Subsection~\ref{milmono}).
On the other hand, the cohomology group of $\RGC(V_{\epsilon};\Q)$ is endowed with also Deligne's mixed Hodge structure.
In what follows, we use the symbol $\RGC(f^{-1}(\epsilon);\Q)_{\DEL}$ to express $\RGC(V_{\epsilon};\Q)$ with Deligne's mixed Hodge structure,
and the symbol $\RGC(V_{\epsilon};\Q)_{\LIM}$ to express $\RGC(V_{\epsilon};\Q)$ with the limit mixed Hodge structure.

Taking the virtual Poincar\'e polynomials of the terms in the distinguished triangle (\ref{dt2}), we obtain
\begin{align}\label{siki2}
\VP(\RGC(V,\Q)(T)
=\VP(\RGC(V_{\epsilon};\Q)_{\LIM})(T)-\VP(H^{n-1}(F_0;\Q)[-(n-1)])(T).
\end{align}
Decompose $\VP(\RGC(X;\Q))(T)$ into $\VP(\RGC(V;\Q))(T)$ and $\VP(\RGC(X\setminus V;\Q))(T)$ and by using (\ref{siki2}), we have
\begin{align*}
\VP(\RGC(X;\Q))(T)
&=\VP(\RGC(V;\Q))(T)+\VP(\RGC(X\setminus V;\Q))(T)\notag
\\&=\VP(\RGC(V_{\epsilon};\Q)_{\LIM})(T)-\VP(H^{n-1}(F_0;\Q)[-(n-1)])(T)\notag\\
&\quad+\VP(\RGC(X\setminus V;\Q))(T).
\end{align*}
Putting this into (\ref{tocyuu}), we obtain
\begin{align}\label{VPICsiki}
\VP(\RGC(X;\ICXT))(T)=\VP(\RGC(V_{\epsilon};\Q)_{\LIM})(T)-\VP(H^{n-1}(F_0;\Q)[-(n-1)])(T)
\\+\VP(\RGC(X\setminus V;\Q))(T)-1+\VP(\ICVTS)(T).\notag
\end{align}

First, we examine the first term in this equation. 
We denote by $W_{\bullet}$ Deligne's weight filtration of $H^{n-1}_c(V_{\epsilon};\C)$,
and by $M_{\bullet}$ the relative monodromy filtartion, respectively (see Section~\ref{mhs}).
Then for $q\neq n-1$, we have
\begin{align*}
\dim{{\GR^{M}_{q}}H^{n-1}_c(V_{\epsilon};\C)}
&=\sum_{r\in\Z}{\dim{{\GR^M_{q}}{\GR^W_{r}}H^{n-1}_c(V_{\epsilon};\C)}}
\\&=\dim{{\GR^M_{q}}{\GR^W_{q}}{H^{n-1}_c(V_{\epsilon};\C)}}
\\&\quad+\dim{{\GR^M_{q}}{\GR^W_{n-1}}{H^{n-1}_c(V_{\epsilon};\C)}}
\\&=\dim{{\GR^W_{q}}{H^{n-1}_c(V_{\epsilon};\C)}}
\\&\quad+\dim{{\GR^M_{q}}{\GR^W_{n-1}}{H^{n-1}_c(V_{\epsilon};\C)}}
\end{align*}
Here, the second and third equalities follow from Proposition~\ref{syuucyuu}.
According to the weak Lefschetz type theorem (see \cite[Proposition 3.9]{DK}),
we have $H^{k}_c(V_{\epsilon};\C)=0$ for any $k\neq n-1,2(n-1)$.
The monodromy action on $H^{2(n-1)}_c(V_{\epsilon};\C)\simeq \C$ is trivial,
and hence $\dim{{\GR^M_{2(n-1)}}}H^{2(n-1)}_c(V_{\epsilon};\Q) = 1$ and $\dim{{\GR^M_{q}}}H^{2(n-1)}_c(V_{\epsilon};\Q)=0$ for any $q\neq 2(n-1)$.
We thus obtain
\begin{align}
\VP(\RGC(V_{\epsilon};\Q)_{\LIM})(T)
&=\sum_{i\in\Z}(-1)^{n-1}\dim{\GR^M_{i}}H^{n-1}_c(V_{\epsilon};\Q)T^{i}+T^{2(n-1)}
\\&=\sum_{n-1\neq i\in\Z}(-1)^{n-1}(\dim{\GR^W_{i}}H^{n-1}_c(V_{\epsilon};\Q)\notag
\\&\qquad+\dim{\GR^M_{i}}{\GR^W_{n-1}}H^{n-1}_c(V_{\epsilon};\Q))T^{i}\notag
\\&\quad+(-1)^{n-1}\dim{\GR^M_{n-1}}H^{n-1}_c(V_{\epsilon};\Q)T^{n-1}
+T^{2(n-1)}\notag
\\&=\sum_{i\in\Z}(-1)^{n-1}{\dim{\GR^W_{i}}H^{n-1}_c(V_{\epsilon};\Q)}T^{i}\notag
\\&\quad+T^{2(n-1)}+Q_1(T)\notag
\\&=\VP(\RGC(f^{-1}(\epsilon);\Q)_{\DEL})(T)\notag
+Q_1(T),\notag
\end{align}
where we set
\begin{align*}
Q_1(T) :=& (-1)^{n-1}\sum_{n-1\neq i\in\Z}{\dim{{\GR^M_{i}}{\GR^W_{n-1}}{H^{n-1}_c(V_{\epsilon};\C)}}}T^{i}
\\&+ (-1)^{n-1}\dim{\GR^M_{n-1}}H^{n-1}_c(V_{\epsilon};\Q)T^{n-1}
\\&-(-1)^{n-1}\dim{\GR^W_{n-1}}H^{n-1}_c(V_{\epsilon};\Q)T^{n-1}.
\end{align*}
Since the polynomial $\sum_{i\neq n-1}{\dim{{\GR^M_{i}}{\GR^W_{n-1}}{H^{n-1}_c(V_{\epsilon};\C)}}}T^{i}$ has a symmetry centered at the degree $n-1$ (see Section~\ref{mhs}),
$Q_1(T)$ is a symmetric polynomial centered at the degree $n-1$.
The projectivization of $f-\epsilon\in\C[x_1,...,x_n]$ is a homogeneous polynomial $g-\epsilon x^m_0\in\C[x_0,...,x_n]$.
We denote by $X_{\epsilon}$ the hypersurface in $\Pn$ defined by this polynomial.
Then we have $X_{\epsilon}\cap\C^n=V_{\epsilon}$.
Recall that we decomposed $\Pn$ into small tori.
Since the intersection of each torus and $X_{\epsilon}$ is smooth,
$X_{\epsilon}$ is smooth in $\Pn$.
We define the hyperplane $\mathbb{L}$ of $\Pn$ by $\mathbb{L}:=\{[x_0:\cdots:x_n]\in\Pn\ |\ x_0=0\}\simeq \PR^{n-1}$.
Then we have $\PR^n=\C^n\sqcup\mathbb{L}$ and $X_{\epsilon}\cap\mathbb{L}=X\cap\mathbb{L}$.
By $f^{-1}(\epsilon)\sqcup (X\cap\mathbb{L})=X_{\epsilon}$, we have
\begin{align*}
\VP(\RGC(f^{-1}(\epsilon);\Q)_{\DEL})(T)+\VP(\RGC(X\cap\mathbb{L};\Q))(T)=\VP(\RGC(X_{\epsilon};\Q))(T).
\end{align*}
We set $Q_{2}(T):=\VP(\RGC(X_{\epsilon};\Q))(T)$.
Since $X_{\epsilon}$ is a smooth hypersurface in $\Pn$, the polynomial $Q_{2}(T)$ has a symmetry centered at the degree $n-1$ by the Poincar\'{e} duality.
Note that $X\cap\mathbb{L}=X\setminus V$.
Thus the sum of the first and third terms of the right hand side of the equation (\ref{VPICsiki}) is calculated as:
\begin{align*}
&\VP(\RGC(V_{\epsilon};\Q)_{\LIM})(T)+\VP(\RGC(X\cap\mathbb{L};\Q))(T)
\\&=\VP(\RGC(X_{\epsilon};\Q))(T)+Q_1(T)\notag
\\&=Q_{2}(T)+Q_{1}(T).\notag
\end{align*}

Next, we examine the second term of the right hand side of the equation (\ref{VPICsiki}).
For a Laurent polynomial $\VP(T)=\sum_{i\in\Z}a_{i}T^{i}\in\Z[T^{\pm}]$ of one variable with coefficients in $\Z$ and
any integer $i_{0}\in\Z$, 
we define the two polynomials:
\begin{align*}
\trun_{\leq i_{0}}\VP(T):=\sum_{i\leq i_{0}}a_{i}T^{i}\text{ \ and \ } \trun_{\geq i_{0}}\VP(T):=\sum_{i\geq i_{0}}a_{i}T^{i}.
\end{align*}
Let $Q_{3}(T)\in\Z[T]$ be the symmetric polynomial centered at $n-1$ satisfying the condition
\[\trun_{\leq n-1}{Q_{3}}(T)=(-1)^{n-1}\sum_{\lambda\neq1}(\sum_{i=0}^{n-1}(\sum_{s\geq0}J^{\lambda}_{n-i+2s})T^i).\]
Moreover we also define the symmetric polynomial $Q_{4}(T)\in\Z[T]$ centered at $n$
by the condition
\[\trun_{\leq n}{Q_{4}}(T)=(-1)^{n-1}\sum_{i=2}^{n}(\sum_{s\geq0}J^{1}_{n+1-i+2s})T^i.
\]
By Proposition~\ref{milnohyou}, we have
\begin{align*}
\VP(H^{n-1}(F_0;\Q)[-(n-1)])(T)=Q_{3}(T)+Q_{4}(T).
\end{align*}
Then we can rewrite the equation (\ref{VPICsiki}) as,
\begin{align}\label{siki7}
\VP(\RGC(X;\ICXT))(T)
&=Q_{1}(T)+Q_{2}(T)-(Q_{3}(T)+Q_{4}(T))
\\&-1+\VP(\ICVTS)(T).\notag
\end{align}
As we saw in Section~\ref{icandmilnor},
the polynomial $\VP(\ICVTS)(T)$ has a degree $\leq n-2$.
Thus we have
\begin{align*}
&\trun_{\geq n-1}\VP(\RGC(X;\ICXT))(T)&
\\&=\trun_{\geq n-1}(Q_{1}(T)+Q_{2}(T)-Q_{3}(T)-Q_{4}(T)).
\end{align*}
Since $\VP(\RGC(X;\ICXT))(T)$ has a symmetry centered at the degree $n-1$
by the generalized Poincar\'e duality,
we also have
\begin{align}\label{siki8}
&\trun_{\leq n-1}\VP(\RGC(X;\ICXT))(T)
\\&=T^{2(n-1)}\trun_{\geq n-1}(Q_{1}(T')+Q_{2}(T')-Q_{3}(T')-Q_{4}(T'))|_{T'=T^{-1}}.\notag
\end{align}
Then by taking the truncations $\trun_{\leq {n-1}}(*)$ of the both sides of (\ref{siki7}),
we obtain
\begin{align*}
&\VP(\ICVTS)(T)
\\&\quad=1+T^{2(n-1)}\trun_{\geq n-1}(Q_{1}(T')+Q_{2}(T')-Q_{3}(T')-Q_{4}(T'))|_{T'=T^{-1}}
\\&\quad\quad-\trun_{\leq n-1}(Q_1(T)+Q_2(T)-Q_3(T)-Q_4(T))
\\&\quad=1+\trun_{\geq n-1}(-Q_4(T'))|_{T'=T^{-1}}-\trun_{\leq n-1}(-Q_4(T))
\\&\quad=1+(-1)^{n-2}\sum_{s=0}^{n-2}J^{1}_{n-s+1}T^s ,
\end{align*}
where the second equality follows from the symmetries centered at the degree $n-1$ of $Q_1(T),Q_2(T),Q_3(T)$.
This completes the proof of Lemma~\ref{mainlem}.
\end{proof}

As we will see in the following Proposition~\ref{subetedoukei},
the coefficients of $f$ in $\Gamma_{+}(f)\setminus \Gamma_{f}$ do not affect the mixed Hodge structures of the IC stalk.
Fix a polytope $P\subset \Gamma_{+}(f)$ such that the union of bounded faces of
$\Conv (\bigcup_{\alpha\in P\cap\Z^n}(\alpha+\R^{n}_{\geq 0}))$ is equal to $\Gamma_{f}$. 
Set $l:=\#\{P \cap\Z^n\}$ and identify $\C^{l}$ with $\C^{P\cap\Z^n}$.
For $y=(y_{\alpha})\in\C^{P\cap\Z^n}$,
set $\sigma_{y}f:=\sum_{\alpha}(a_{\alpha}+y_{\alpha})x^{\alpha}$ (when $\alpha\not\in P$, we set $y_{\alpha}=0$)
and denote by $V_{y}$ the zero set of $\sigma_{y}f$ in $\C^n$.
Let $\Omega\subset\C^l$ be the subset of $\C^l$ consisting $y\in\C^{l}$
such that $\Gamma_{\sigma_{y}f}$ is equal to $\Gamma_{f}$ and $\sigma_{y}f$ is non-degenerate at $0\in\C^{n}$.
Note that $\Omega$ is a Zariski open and hence path-connected subset of $\C^{l}$.
Under this setting, we have the following result.

\begin{prop}\label{subetedoukei}
For any $y\in \Omega$ and $k\in \Z$,
the mixed Hodge numbers of $H^{k}((\IC_{V_{y}})_0)$ is the same as those of $H^{k}((\ICV)_{0})$.
\end{prop}
This Proposition will be proved in Section~\ref{sechenkei}.
By Proposition~\ref{subetedoukei}, the proof of Theorem~\ref{mainth} is reduced to the case in Lemma~\ref{mainlem} as follows. 

\begin{proof}[Proof of Theorem~\ref{mainth}]
If the polynomial $f$ has some linear terms, $V$ is non-singular at $0\in V$ and $\Phi_{0,n-1}=0$.
Thus, our assertion is trivial.
Therefore, in what follows, we assume that the polynomial $f$ has no linear term.
Since the non-degeneracy condition is a generic condition for each fixed Newton polytope,
there exist a polytope $P\subset \R^n_{\geq 0}$ and $y\in \Omega$ 
for which $\sigma_{y}f$ satisfies the conditions of Lemma~\ref{mainlem}.
By Proposition~\ref{subetedoukei},
the mixed Hodge numbers of $\wt{(\mathrm{IC}_{V_{y}})}_0$ and $\ICVTS$ are the same.
On the other hand, since both $\sigma_{y}f$ and $f$ are non-degenerate at $0$, convenient and satisfy the condition $\Gamma_{\sigma_{y}f}=\Gamma_{f}$,
the Jordan normal forms of the Milnor monodromies of $\sigma_{y}f$ and $f$ are the same (see Theorem~\ref{jordangammadekimaru}).
Then the assertion immediately follows from Lemma~\ref{mainlem}.
\end{proof}

For $n=2$, similarly we obtain the following theorem.
\begin{theo}\label{n2case}
Assume that $n=2$ and $f$ is convenient and non-degenerate at $0$.
Then for any $r\in\Z$, we have
\begin{align*}
\dim \GR^{W}_{r}H^{0}((\ICVT)_0) =
\left\{\renewcommand{\arraystretch}{1.5}
\begin{array}{ll}
N_{0}+1&\text{ if\ \ }r=0,\text{ and}\\
0&\text{ if\ \ }r\neq 0.
\end{array}
\right.
\end{align*}
\end{theo}
\renewcommand{\arraystretch}{1}

\section{Proof of Proposition~\ref{subetedoukei}}\label{sechenkei}
In this section, we prove Proposition~\ref{subetedoukei}.
Let $f=\sum_{\alpha\in\R^n_{\geq 0}}a_{\alpha}x^{\alpha}\in\C[x_{1},\dots,x_{n}]$ be a non-constant polynomial of $n$ variables with coefficients in $\C$ such that $f(0)=0$.
Assume that $f$ is convenient and non-degenerate at $0$,
and the hypersurface $V=\{x\in\C^n\ |\ f(x)=0\}$ has an isolated singular point $0$.
In what follows, we use the notations in Proposition~\ref{subetedoukei},
such as $P$, $l$, $\Omega$, $\sigma_{y}f$, $V_{y}$ $(y=(y_{\alpha})\in\C^{l}\simeq\C^{P\cap\Z^n})$, etc.
Define a polynomial $\wt{f} \in \C[ x_1, \dots, x_n, (y_{\alpha})_{\alpha\in P\cap\Z^n}]$ of $n+l$ variables by $\wt{f}(x_1,\dots,x_n,(y_{\alpha})_{\alpha\in P\cap\Z^n}):=\sum_{\alpha}(a_{\alpha}+y_{\alpha})x^{\alpha}$
(if $\alpha\not\in P$, we set $y_{\alpha}=0$).
We denote by $\wt{V}$ the zero set of $\wt{f}$ in $\C^{n+l}$.
Note that $\sigma_{y}f\in\C[x_{1},\dots,x_{n}]$ is identified with $\wt{f}(x_{1},\dots,x_{n},(y_{\alpha}))$ for any fixed element $y=(y_{\alpha})\in\C^{l}$.

\begin{lem}\label{koritu}
For any compact subset $K$ of $\Omega$,
there is a neighborhood $U$ of $0$ in $\C^n$
such that for any $y\in K$, 
the origin $0\in V_{y}$ is the only singular point of $V_{y}$ in $U$.
\end{lem}
\begin{proof}
Fix a compact set $K$ in $\Omega$.
We denote by $\Sigma_{0}$ the dual fan of $\Gamma_{+}(f)$ in $\R^n$.
Take a smooth subdivision $\Sigma$ of the fan $\Sigma_{0}$ without subdividing the cones in the boundary of $(\R_{\geq0})^n$ in $\R^n$.
Let $\Sigma'$ be the fan formed by all faces of $(\R_{\geq0})^n$ in $\R^n$.
We denote by $X_{\Sigma}$ and $X_{\Sigma'}$ the smooth toric varieties associated to $\Sigma$ and $\Sigma'$ respectively.
Note that $X_{\Sigma'}$ is isomorphic to $\C^n$.
The identity map $\R^n\to\R^n$ induces a morphism of fans from $\Sigma$ to $\Sigma'$,
and we obtain a proper morphism $\pi\colon X_{\Sigma}\to\C^n$.
Then $\pi$ induces an isomorphism $X_{\Sigma}\setminus \pi^{-1}(0)\simeq \C^n\setminus\{0\}$.
Suppose that there exits a sequence $(y_{i})_{i\in\N}$ in $K$ and a singular points $x_{i}$ of  $V_{y_{i}}\cap\C^n$ which satisfy $0<|x_{i}|<1/(i+1)$.
By using the isomorphism $X_{\Sigma}\setminus \pi^{-1}(0)\simeq \C^n\setminus\{0\}$, we can regard $x_{i}$ as an element of $X_{\Sigma}$.
Since $K$ is compact, we may assume that the sequence $(y_{i})_{i\in\N}$ converges to a point $y_{\infty}\in K$.
Moreover since the pullback of the closed ball $\overline{B(0,1)}$ in $\C^n$ by the proper map $\pi$ is compact, we may assume $(x_{i})_{i\in\N}$ converges to a certain point $x_{\infty} \in X_{\Sigma}$.
It is clear that $x_{\infty}\in \pi^{-1}(0)$.
Since $\sigma_{y_{\infty}}f$ is non-degenerate at $0$,
the hypersurface in $X_{\Sigma}$ defined by the pull back of $V_{y_{\infty}}$ by $\pi$
intersect with $\pi^{-1}(0)$ transversely.
However, one can easily see that
$x_{\infty}$ is a singular point of $\pi^{-1}(V_{y_{\infty}})\cap\pi^{-1}(0)$.
This is a contradiction and completes the proof.
\end{proof}

The next lemma is a modified version of the cone theorem \cite[Thorem~2.10]{MIL}
with some parameters.
For $\epsilon>0$, we denote by $D(0,\epsilon)$ the closed ball in $\C^n$ centered at $0$ with radius $\epsilon$, by $\partial D(0,\epsilon)$ its boundary,
and set $D(0,\epsilon)^{*}=D(0,\epsilon)\setminus\{0\}$.

\begin{lem}[cf. {\cite[Theorem~2.10]{MIL}}]\label{cone}
For a relatively compact open subset $U$ of $\Omega$,
$\wt{V}\cap (B(0,\epsilon)^{*}\times{U})$ is homotopic to $\wt{V}\cap (\partial D(0,\epsilon)^{*}\times U)$.
If moreover $U$ is contractible, 
$\wt{V}\cap (B(0,\epsilon)^{*}\times{U})$ is homotopic to $(V_{y}\cap\partial{D(0,\epsilon)})\times \{y\}$ for any point $y\in U$.
\end{lem}
\begin{proof}
By using Lemma~\ref{koritu}, we can apply the same argument in the proof of \cite[Corollary~2.10]{MIL}.
\end{proof}

Next, we consider the intersection cohomology complex $\IC_{\widetilde{V}}$ of $\widetilde{V}$.
Note that $\IC_{\widetilde{V}}$ is the underlying constructible sheaf of the mixed Hodge module $\IC_{\widetilde{V}}^{H}$.
\begin{prop}\label{teisuusou}
The restriction of $\IC_{\wt{V}}$ to $\{0\}\times \Omega$ is cohomologically locally constant,
that is, for any $k\in\Z$ the cohomology group $H^{k}(\IC_{\wt{V}}|_{\{0\}\times \Omega})$ is a locally constant sheaf on $\Omega$.
\end{prop}
\begin{proof}
The smooth part of ${\widetilde{V}}$ is ${\widetilde{V}}_{\REG}={\widetilde{V}}\cap((\C^n\setminus\{0\})\times \Omega)$.
We set $j:={\widetilde{V}}_{\REG}\hookrightarrow{\widetilde{V}}$.
Since $\IC_{\widetilde{V}}=\tau^{\leq -1}(\DER{j}_{*}{\Q_{\widetilde{V}_{\REG}}[n+l-1]})$, 
it is sufficient to show that each cohomology sheaf of the restriction of $\DER{j}_{*}{\Q_{\widetilde{V}_{\REG}}}$ to $\{0\}\times \Omega$ is locally constant sheaf.

Take any closed ball in $K\subset \Omega$ and a point $y\in K$.
For any integer $k\in\Z$, we have
\begin{align*}
H^{k}(\wt{V}\cap(\{0\}\times K);\DER{j}_{*}{\Q_{\wt{V}_{\REG}}})\simeq
\varinjlim_{\substack{0<\epsilon \ll 1\\ K\subset U}}{H^{k}(\wt{V}\cap(B(0,\epsilon)^{*}\times U);\Q)},
\end{align*} 
where $U$ ranges through the family of open contractible neighborhoods of $K$ in $\Omega$.
By Lemma~\ref{cone}, $\wt{V}\cap (B(0,\epsilon)^{*}\times U)$ is homotopic to
$(V_{y}\cap \partial D(0,\epsilon))\times \{y\}$.
We thus obtain
\[H^{k}(\wt{V}\cap(\{0\}\times K);\DER{j}_{*}{\Q_{\wt{V}_{\REG}}})\simeq
H^{k}((V_{y}\cap \partial D(0,\epsilon))\times \{y\};\Q),\]
for a sufficiently small $\epsilon>0$.
On the other hand, the stalk of $H^{k}(\DER{j}_{*}{\Q_{\widetilde{V}_{\REG}}})$ at $(0,y)$ is isomorphic to
\[\varinjlim_{\substack{0<\epsilon\ll1\\y\in U}}{H^{k}(\wt{V}\cap(B(0,\epsilon)^{*}\times U);\Q)},\]
where $U$ ranges through the family of open contractible neighborhoods in $\Omega$ of $y$.
By Lemma~\ref{cone}, $\wt{V}\cap(B(0,\epsilon)^{*}\times U)$ is homotopic to $(V_{y}\cap\partial D(0,\epsilon))\times \{y\}$.
Hence
$H^{k}((\DER{j}_{*}{\Q_{\widetilde{V}_{\REG}}})_{(0,y)})$ is isomorphic to $H^{k}((V_{y}\cap\partial D(0,\epsilon))\times \{y\};\Q)$,
for a sufficiently small $\epsilon>0$.
Eventually, the morphism induced by the inclusion map from $(0,y)$ into $\wt{V}\cap(\{0\}\times K)$,
\[H^{k}(\wt{V}\cap(\{0\}\times K);\DER{j}_{*}{\Q_{\wt{V}_{\REG}}}|_{\wt{V}\cap(\{0\}\times K)})
\to
H^{k}((0,y);\DER{j}_{*}{\Q_{\wt{V}_{\REG}}}|_{\wt{V}\cap(\{0\}\times K)}),\]
is isomorphic.
This shows that $\DER{j}_{*}{\Q_{\wt{V}_{\REG}}}|_{\wt{V}\cap(\{0\}\times \Omega)}$ is a cohomologically locally constant sheaf.
\end{proof}

Finally, we shall prove Proposition~\ref{subetedoukei}, by using Lemma~\ref{teisuusou}.

\begin{proof}[Proof of Proposition~\ref{subetedoukei}]
We consider the restriction of the complex of mixed Hodge modules $\IC_{\wt{V}}^H|_{\{0\}\times \Omega}\in \DCB\MHM(\Omega)$ of $\IC_{\wt{V}}^H$,
and its cohomology groups of $\IC_{\wt{V}}^H|_{\{0\}\times \Omega}$.
The underlying perverse sheaf of these mixed Hodge modules on $\Omega$
is each perverse cohomology of $\IC_{\wt{V}}|_{\{0\}\times \Omega}$.
However in this situation,
the $k$-th perverse cohomology is the $(k-l)$-th usual cohomology, which is a locally constant sheaf by Lemma~\ref{teisuusou}.  

Fix a degree $k\in\Z$.
We denote by $(\mathcal{M},F^{\bullet},K,W_{\bullet})$ the mixed Hodge module $H^{k}(\IC_{\wt{V}}^H|_{\{0\}\times \Omega})$ on $\Omega$.
Here $\mathcal{M}$ is a regular holonomic $D_{\Omega}$-module on $\Omega$, $F^{\bullet}$ is a good filtration of $\mathcal{M}$, $K$ is the underlying perverse sheaf of $\mathcal{M}$ and
$W_{\bullet}$ is a finite increasing filtration of $(\mathcal{M},F^{\bullet},K)$ in $\mathrm{MF}_{\mathrm{rh}}(D_{\Omega},\Q)$.
From the above argument, $K$ is a shifted locally constant sheaf.
Thus, $\mathcal{M}$ associated to $K$ is an integrable connection.  
Hence all $D_{\Omega}$-submodules of $\mathcal{M}$ in the filtration $W_{\bullet}$ are coherent $\OR_{\Omega}$-modules,
and hence integrable connections (see \cite{HTT}, etc.).
Therefore, $(\mathcal{M},F^{\bullet},K,W_{\bullet})$ is a mixed Hodge module associated to a variation of mixed Hodge structures on the connected open subset $\Omega\subset \C^l$.
Since the mixed Hodge numbers at all points in $\Omega$ of this variation of mixed Hodge structures are the same,
this completes the proof. 
\end{proof}

\bibliography{myref}

\end{document}